\documentclass[a4paper, 12pt]{amsart}
\usepackage[mathscr]{eucal}
\usepackage{amssymb}
\usepackage{amsthm}
\usepackage{amsfonts}
\usepackage{xcolor}
\usepackage[breaklinks=true,colorlinks=true,linkcolor=blue,citecolor=red,urlcolor=green]{hyperref}
\usepackage{setspace}

\usepackage[margin=1in,footskip=0.25in]{geometry}

\usepackage{color}

\numberwithin{equation}{section}

\usepackage{lineno}
\usepackage{graphicx}
\allowdisplaybreaks
\usepackage{hyperref}
\hypersetup{
colorlinks=true,
urlcolor=black,
citecolor=blue}

\theoremstyle{plain}
\newtheorem{theorem}{Theorem}[section]
\newtheorem{proposition}[theorem]{Proposition}

\newtheorem{corollary}[theorem]{Corollary}

\theoremstyle{definition}

\newtheorem{example}[theorem]{Example}

\theoremstyle{remark}
\newtheorem{remark}[theorem]{Remark}

\usepackage{amsmath,amssymb}
\usepackage{booktabs}
\usepackage{tabularx}
\usepackage{array}
\numberwithin{equation}{section}
\usepackage{makecell}

\newcolumntype{C}[1]{>{\centering\arraybackslash}m{#1}}
\newcolumntype{L}[1]{>{\raggedright\arraybackslash}m{#1}}

\begin{document}

\title[Weakly Einstein and weakly $\eta$-Einstein structures]{
Weakly Einstein and weakly $\eta$-Einstein structures on almost contact metric three-manifolds}

\author[S. H. Chun]{Sun Hyang Chun}
\address{Sun Hyang Chun \\ Department of Mathematics \\ Chosun University \\ Gwangju 61452, Korea}
\email{shchun@chosun.ac.kr}
\thanks{S. H. Chun was supported by research funds from Chosun University, 2026.}

\author[Y. Euh]{Yunhee Euh}
\address{Yunhee Euh \\ Department of Mathematics \\ Sungkyunkwan University \\ Suwon 16419, Korea}
\email{yunhee.euh@gmail.com}
\thanks{Y. Euh  was supported by Basic Science Research Program through the National Research Foundation of Korea (NRF) funded by the Ministry of Education(Grant No. RS-2023-00244736).}


\begin{abstract}
We study weakly Einstein and weakly $\eta$-Einstein structures on
almost contact metric three-manifolds. We first obtain a pointwise
structural description of the Ricci operator $Q$ on a weakly
$\eta$-Einstein almost contact metric three-manifold. More precisely,
at each point, either $Q$ has the $\eta$-Einstein form, $Q\xi=0$, or
the metric is weakly Einstein. We then prove that if a contact metric
three-manifold satisfies $Q\xi=0$, then either $Q=0$ or
$\operatorname{rank}Q=1$. Hence, such a metric is either flat or
weakly Einstein but not Einstein. We also construct examples showing
that the contact metric assumption is essential: there exist almost
contact metric three-manifolds satisfying $Q\xi=0$ whose Ricci
operators have rank two. Finally, we classify simply connected
homogeneous contact metric three-manifolds that are weakly Einstein.
\end{abstract}

\maketitle

\section{Introduction}\label{sec:intro}

Almost contact metric manifolds are Riemannian manifolds endowed with
structure tensors $(\varphi,\xi,\eta,g)$ satisfying identities compatible with the metric. The defining identities of an almost contact metric structure imply the orthogonal decomposition
$$  TM=D\oplus \mathbb R\xi,\qquad D=\ker\eta.$$ 
Thus, on such a manifold, curvature tensors may be studied together with the distribution $D$ and the line field $\mathbb{R}\xi$ generated by the characteristic vector field $\xi$.

A Riemannian manifold $(M^m,g)$ is called \textit{weakly Einstein} if
$$  \breve R=\frac{|R|^2}{m}g,$$
where $  \breve R(X,Y)=\sum_{i,j,k=1}^m  R(X,e_i,e_j,e_k)R(Y,e_i,e_j,e_k)$ for a local orthonormal frame $\{e_i\}$. This notion was introduced in the study of curvature identities on four-dimensional Riemannian manifolds \cite{EPS2013}. In dimension three, the Weyl tensor vanishes identically, and hence the Riemannian curvature tensor is determined by the Ricci tensor. Therefore the weakly Einstein equation can be expressed
in terms of the Ricci operator $Q$. More precisely, a  three-dimensional Riemannian manifold is weakly Einstein if and only if, at each point, $Q$ is a scalar multiple of the identity, or $Q$ has rank one in the
 non-Einstein case \cite{GHV18}.

Recent work on low-dimensional weakly Einstein geometry shows that the weakly Einstein condition is particularly effective in classification and rigidity problems. Classification and rigidity results have been obtained for four-dimensional homogeneous weakly Einstein manifolds \cite{AMK15}, locally conformally flat weakly Einstein manifolds \cite{GHV18,MV23}, weakly Einstein Kähler surfaces \cite{DEKP25}, weakly Einstein Lie groups \cite{EKNP25}, and weakly Einstein submanifolds and hypersurfaces in space forms \cite{KP24, KNP26}. In almost contact metric geometry, weakly Einstein metrics have also been investigated for Sasakian manifolds, contact $(\kappa,\mu)$-spaces, and almost cosymplectic manifolds \cite{Chen2020}. These results indicate that the weakly Einstein condition is flexible enough to admit non-Einstein examples, but still rigid enough to yield classification results in special geometric settings.

In \cite{CCE22},  Cho, Chun, and Euh  introduced the notion of a weakly $\eta$-Einstein structure on an almost contact metric manifold, defined by the condition
\begin{equation*}
\breve{R}=\alpha g+\beta\,\eta\otimes\eta,
\end{equation*}
where $\alpha$ and $\beta$ are smooth functions. In subsequent work, we investigated weakly $\eta$-Einstein homogeneous contact metric three-manifolds and several classes of weakly $\eta$-Einstein almost contact metric three-manifolds, respectively \cite{CE23,CE25}.

In this paper, we  study weakly Einstein and weakly $\eta$-Einstein almost contact metric three-manifolds by using the Ricci operator. A basic point is that a local $\varphi$-orthonormal frame does not necessarily diagonalize the full Ricci operator. Although $Q$ is symmetric on $TM$, it need not
preserve $D=\ker\eta$. Hence one cannot assume that $Q$ is diagonal with respect to a $\varphi$-basis. Instead, one first diagonalizes
$$  Q_D=\pi_D\circ Q|_D,$$
where $\pi_D$ denotes the orthogonal projection from $TM$ onto $D$.
The Ricci components involving the $\xi$-direction remain in the
calculation, and these components are necessary for the weakly
$\eta$-Einstein equation. This gives a precise way to compare the
weakly Einstein equation with the weakly $\eta$-Einstein equation.

The paper is organized as follows. Section~\ref{sec:preliminaries} recalls basic facts on almost contact metric manifolds, weakly Einstein metrics, weakly
$\eta$-Einstein metrics, and the curvature identity in dimension three. We also recall the Ricci-operator characterization of
weakly Einstein metrics in dimension three. Section~\ref{sec:weakly-eta} proves a pointwise result for weakly $\eta$-Einstein almost contact
metric three-manifolds.  At each point, one of the following holds: the Ricci operator has the
$\eta$-Einstein form, $Q\xi=0$, or the metric is weakly Einstein at that point. Section~\ref{sec:weakly-einstein-acm} studies the non-Einstein weakly Einstein case. Since the Ricci operator has rank one, there is a unique one-dimensional eigenspace corresponding to the nonzero Ricci eigenvalue. We describe the possible forms of $Q$ according to whether this eigendirection lies in $D=\ker \eta$, lies along $\mathbb R\xi$, or has nonzero components in both directions. Section~\ref{sec:contact} proves that, on a  contact metric three-manifold, the condition $Q\xi=0$ implies $\operatorname{rank}Q\leq 1$. Consequently, the metric is weakly Einstein. Section~\ref{sec:examples} gives examples showing that the conclusion of Section~\ref{sec:contact} 
does not hold without the contact metric assumption.  We construct an
almost contact metric manifold satisfying $Q\xi=0$ and
$\operatorname{rank}Q=1$, and another one satisfying $Q\xi=0$ and
$\operatorname{rank}Q=2$. Finally, Section~\ref{sec:homogeneous} determines all simply connected
 homogeneous contact metric three-manifolds that are weakly Einstein. 

\section{Preliminaries}\label{sec:preliminaries}

Throughout this paper, all vector fields are assumed to be smooth unless otherwise stated. Let $M^{2n+1}$ be a smooth manifold. An almost contact metric structure on $M$ is a quadruple $(\varphi,\xi,\eta,g)$, where $\varphi$ is a tensor field of type $(1,1)$, $\xi$ is a vector field, $\eta$ is a one-form, and $g$ is a Riemannian metric satisfying $$ \varphi^2=-I+\eta\otimes\xi,\qquad \eta(\xi)=1,\qquad g(\varphi X,\varphi Y)=g(X,Y)-\eta(X)\eta(Y) $$ for all vector fields $X,Y$ on $M$ \cite{Blair2010}. These identities imply $$ \varphi\xi=0,\qquad \eta\circ\varphi=0,\qquad \eta(X)=g(X,\xi). $$

An almost contact metric manifold $M$ such that 
\begin{equation}
    \label{eq:contact}
    d\eta(X, Y)=g(X, \varphi Y)
\end{equation}
is called a contact metric manifold, where $d$ is the exterior derivative. Given a contact metric manifold, we define the structure operator $h$ by $h=\frac{1}{2}\mathcal{L}_{\xi}\varphi$, where $\mathcal{L}_\xi$ is the Lie derivative in the direction of $\xi$. Then we easily see that $h$ is symmetric and satisfies the following conditions:
\begin{equation}
    \label{eq:structure_h}
    h\xi=0,\qquad h\varphi=-\varphi h, \qquad \nabla_X\xi=-\varphi X-\varphi h X.
\end{equation}
We denote by $D=\ker\eta$ the horizontal distribution, or equivalently the $\eta$-distribution.
In contact metric manifolds, $D$ is called the contact distribution.  In dimension three, $D$ has rank two, and locally one may choose a $\varphi$-orthonormal frame $ \{e_1,e_2=\varphi e_1,\xi\}.$

Our convention for the Riemannian curvature tensor is $$ R(X,Y)Z = \nabla_X\nabla_Y Z -\nabla_Y\nabla_X Z -\nabla_{[X,Y]}Z, $$ and the associated $(0,4)$-curvature tensor is defined by $$ R(X,Y,Z,W)=g(R(X,Y)Z,W). $$
The Ricci tensor and the scalar curvature are denoted by $\rho$ and $r$,
respectively.  The Ricci operator $Q$ is defined by
$$
  g(QX,Y)=\rho(X,Y).
$$
Since the Weyl tensor vanishes identically in dimension three, the
curvature tensor of a three-dimensional Riemannian manifold is determined
by the Ricci tensor and the scalar curvature.  With the above convention,
one has
\begin{equation}
\begin{aligned}\label{eq:3d-curvature}
R(X,Y,Z,W)
=\,& \rho(Y,Z)g(X,W)-\rho(X,Z)g(Y,W) \\
&+g(Y,Z)\rho(X,W)-g(X,Z)\rho(Y,W) \\
&-\frac r2
\{g(Y,Z)g(X,W)-g(X,Z)g(Y,W)\}.
\end{aligned}
\end{equation}

For a Riemannian manifold $(M^m,g)$, define a symmetric $(0,2)$-tensor
$\breve R$ by
$$
  \breve R(X,Y)
  =
  \sum_{i,j,k=1}^m
  R(X,e_i,e_j,e_k)R(Y,e_i,e_j,e_k),
$$
where $\{e_i\}_{i=1}^m$ is a local orthonormal frame.  This definition is
independent of the choice of the orthonormal frame.  The manifold is
called \emph{weakly Einstein} if  its curvature tensor satisfies
$$
  \breve R=\frac{|R|^2}{m}g.
$$
Equivalently, $\breve R=\alpha g$ for some smooth function $\alpha$.
Indeed, tracing this equation gives $\alpha=|R|^2/m$.

An almost contact metric manifold $(M^{2n+1},\varphi,\xi,\eta,g)$ is
called \emph{weakly $\eta$-Einstein} if there exist smooth functions
$\alpha$ and $\beta$ such that
\begin{equation}
     \label{eq:wee}
  \breve R=\alpha g+\beta\,\eta\otimes\eta.
\end{equation}
Equivalently,
\begin{equation*}
 \breve R(X,Y)=\alpha g(X,Y),\qquad
  \breve R(X,\xi)=0
\end{equation*}
for all $X,Y\in D=\ker\eta$, and
$$
  \breve R(\xi,\xi)=\alpha+\beta.
$$
In particular, every weakly Einstein almost contact metric manifold is
weakly $\eta$-Einstein with $\beta=0$.

An almost contact metric manifold is called \emph{$\eta$-Einstein} if its Ricci tensor satisfies
$$
  \rho=\overline{\alpha}g+\overline{\beta}\,\eta\otimes\eta
$$
for some smooth functions $\overline{\alpha}$ and $\overline{\beta}$.
In dimension three, this is equivalent to saying that, with respect to a
local $\varphi$-orthonormal frame
$\{e_1,e_2=\varphi e_1,\xi\}$, the Ricci operator has the form
$$
  Qe_1=\overline{\alpha} e_1,\qquad
  Qe_2=\overline{\alpha} e_2,\qquad
  Q\xi=(\overline{\alpha}+\overline{\beta})\xi.
$$
Equivalently, $Q|_D=\overline{\alpha}I_D$, and $Q$ has no mixed
components between $D$ and $\mathbb R\xi$.

\begin{proposition}
    \label{rem:eta-implies-weakly-eta}{\cite{CCE22}}
Every $\eta$-Einstein almost contact metric three-manifold
is weakly $\eta$-Einstein.
\end{proposition}

Using the three-dimensional curvature identity \eqref{eq:3d-curvature},
a direct computation gives
\begin{equation}\label{eq:3-d_breve}
  \breve R(X,Y)
  =
  (2|\rho|^2-r^2)g(X,Y)+2r\rho(X,Y)-2g(QX,QY).
\end{equation}
Now let $\{E_1,E_2,E_3\}$ be a local orthonormal frame consisting of
eigenvectors of $Q$, and write $QE_i=\lambda_iE_i$.  Then
\eqref{eq:3-d_breve} gives
$$
  \breve R(E_i,E_i)
  =
  2|\rho|^2-r^2+2r\lambda_i-2\lambda_i^2.
$$
Hence, for mutually distinct indices $i,j,k\in\{1,2,3\}$,
\begin{equation}\label{eq:breve-difference}
  \breve R(E_i,E_i)-\breve R(E_j,E_j)
  =
  2(\lambda_i-\lambda_j)\lambda_k.  
\end{equation}

\begin{proposition}[Lemma~5 in \cite{GHV18}]
\label{prop:3d-we-spectrum}
Let $(M^3,g)$ be a three-dimensional Riemannian manifold.  Then
$(M^3,g)$ is weakly Einstein if and only if, at each point, the Ricci
operator $Q$ is either a scalar multiple of the identity or has rank one.  Equivalently, the eigenvalues of $Q$ are
either
$$
  (\lambda,\lambda,\lambda),\qquad \lambda\in\mathbb R,
$$
or, up to permutation,
$$
  (\lambda,0,0),\qquad \lambda\neq0.
$$
\end{proposition}

\section{Weakly $\eta$-Einstein almost contact metric three-manifolds}
\label{sec:weakly-eta}

In this section, we analyze the weakly $\eta$-Einstein condition in
dimension three in terms of the Ricci operator. With respect to the
orthogonal decomposition $TM = D\oplus \mathbb R\xi$, $D=\ker\eta$, the Ricci operator need not preserve $D$. Thus we consider the projected operator
$  Q_D=\pi_D\circ Q|_D $
on $D$. The following result is pointwise. It describes the role of the mixed
Ricci components between $D$ and $\mathbb R\xi$ in the weakly
$\eta$-Einstein equation.

\begin{theorem}\label{thm:trichotomy}
Let $(M^3,\varphi,\xi,\eta,g)$ be an almost contact metric
three-manifold. Suppose that $M$ is weakly $\eta$-Einstein. Then, at each point $p\in M$, at least one of the following holds.
\begin{enumerate}
\item[\rm(i)] the Ricci operator has the $\eta$-Einstein form ($Q=\overline{\alpha}I+\overline{\beta}\,\eta\otimes\xi$);
\item[\rm(ii)] $Q\xi=0$;
\item[\rm(iii)] the metric is weakly Einstein.
\end{enumerate}
\end{theorem}

\begin{proof}
Fix a point $p\in M$. All tensors, functions, and equations in the
argument below are evaluated at $p$, unless otherwise stated. Since $Q$ is symmetric,
the operator
$$
  Q_D=\pi_D\circ Q|_{D_{p}}
$$
is symmetric on $D_p$, where $\pi_D:T_pM\to D_p$ denotes the orthogonal
projection.  Hence we may choose a local $\varphi$-orthonormal frame $\{e_1,e_2=\varphi e_1,\xi\}$
such that $Q_D$ is diagonal at $p$.  With respect to this frame, the
Ricci operator has the form
$$
  Q=
  \begin{pmatrix}
  \lambda_1 & 0 & a \\
  0 & \lambda_2 & b \\
  a & b & \lambda_3
  \end{pmatrix}.
$$
Using \eqref{eq:wee} and \eqref{eq:3-d_breve}, we obtain
\begin{equation}\label{eq:weakly-eta-expanded}
  (2|\rho|^2-r^2)g(X,Y)+2r\rho(X,Y)-2g(QX,QY)
  =
  \alpha g(X,Y)+\beta\eta(X)\eta(Y).
\end{equation}
Evaluating \eqref{eq:weakly-eta-expanded} with respect to the above
frame, the diagonal components give
\begin{align}
  &\breve R(e_1,e_1)=\lambda_1^2+(\lambda_2-\lambda_3)^2+2a^2+4b^2
  =\alpha, \label{eq:diag1}\\
  &\breve R(e_2,e_2)=(\lambda_1-\lambda_3)^2+\lambda_2^2+4a^2+2b^2
  =\alpha, \label{eq:diag2}\\
  &\breve R(\xi,\xi)=(\lambda_1-\lambda_2)^2+\lambda_3^2+2a^2+2b^2
  =\alpha+\beta. \label{eq:diag3}
\end{align}
The off-diagonal components give
\begin{equation}\label{eq:off-diagonal}
  \breve{R}(e_1,e_2)=-2ab=0,\qquad \breve{R}(e_1,\xi)=2\lambda_2a=0,\qquad \breve{R}(e_2,\xi)=2\lambda_1b=0.
\end{equation}
Moreover, from \eqref{eq:diag1} and  \eqref{eq:diag2}, we have
\begin{equation}\label{eq:horizontal}
  \lambda_3(\lambda_1-\lambda_2)-a^2+b^2=0.
\end{equation}
From \eqref{eq:off-diagonal}, we can consider the following cases:

\smallskip
\noindent\textbf{Case I.} Suppose that $a=b=0$.  Then $\xi$ is an
eigenvector of $Q$ and from \eqref{eq:horizontal}, we obtain 
$$
  \lambda_3(\lambda_1-\lambda_2)=0.
$$
If $\lambda_1=\lambda_2$, then, for every $X\in T_pM$,
$$
  QX=\lambda_1X+(\lambda_3-\lambda_1)\eta(X)\xi.
$$
Thus the Ricci operator has the $\eta$-Einstein form at $p$. If $\lambda_3=0$, then $Q\xi=0$ at $p$.

\smallskip
\noindent\textbf{Case II.} Suppose that exactly one of $a$ and $b$ is
nonzero.

First assume that $a=0$ and $b\neq0$. From
\eqref{eq:off-diagonal}, we have $\lambda_1=0$. Hence, evaluating
\eqref{eq:3-d_breve} at $(e_2,e_2)$ and at $(\xi,\xi)$, we obtain
$$
  \breve R(e_2,e_2)
  =
  \lambda_2^2+\lambda_3^2+2b^2,
  \qquad
  \breve R(\xi,\xi)
  =
  \lambda_2^2+\lambda_3^2+2b^2.
$$
Thus
$
  \breve R(e_2,e_2)=\breve R(\xi,\xi).
$
On the other hand, since $e_1,e_2\in D=\ker\eta$, the weakly
$\eta$-Einstein equation gives
$
  \breve R(e_1,e_1)=\breve R(e_2,e_2).
$
Consequently,
$$
  \breve R(e_1,e_1)=\breve R(e_2,e_2)=\breve R(\xi,\xi).
$$
Hence $\beta=0$. Therefore $\breve R=\alpha g$. Taking the trace
gives $\alpha=|R|^2/3$, and so the metric is weakly Einstein at $p$. 
The case $b=0$ and $a\neq0$ is also treated similarly. 
\end{proof}

\begin{remark}\label{rem:comparison-CE23}
Theorem~\ref{thm:trichotomy} is a refinement of
\cite[Theorem~3.1]{CE23}.  There, it was shown that a weakly
$\eta$-Einstein almost contact metric three-manifold either has Ricci
operator of the $\eta$-Einstein form or satisfies $Q\xi=0$.
The main point of the present theorem is that, on an  almost
contact metric three-manifold, a local $\varphi$-orthonormal frame
need not  diagonalize the Ricci operator $Q$.
Thus, we diagonalize only the horizontal part
$
Q_D=\pi_D\circ Q|_D
$
and keep the mixed Ricci components between $D=\ker\eta$ and
$\mathbb R\xi$. This yields the third case stated
in Theorem~\ref{thm:trichotomy}.
\end{remark}

\section{Weakly Einstein almost contact metric three-manifolds}
\label{sec:weakly-einstein-acm}
We now analyze the weakly Einstein case stated in Theorem~\ref{thm:trichotomy}.
It is known that  a non-Einstein
weakly Einstein metric has the Ricci operator of rank one in dimension three \cite{GHV18}. Therefore,  the nonzero Ricci eigenvalue determines a unique one-dimensional
Ricci eigenspace.  The almost contact metric structure is reflected in the position of this eigenspace relative
to the orthogonal decomposition
$$
  TM=D\oplus\mathbb R\xi.
$$

\begin{theorem}
\label{thm:rank-one-we}
Let $M=(M^3,\varphi,\xi,\eta,g)$ be an almost contact metric
three-manifold. Suppose that  $M$ is weakly Einstein but not Einstein on an open subset $\mathcal{U}$. Then, locally on $\mathcal{U}$, there exist a nonzero smooth function $\lambda$ and a unit vector field $V$ such that
$$
  QX=\lambda g(X,V)V
$$
for all vector fields $X$ on $\mathcal{U}$.  According to the position of $V$ relative to
 $D=\ker\eta$ and $\mathbb R\xi$, one of the following cases
occurs.
\begin{itemize}
\item[\rm(i)]  If $V\in D$, then there exists a local $\varphi$-orthonormal frame $\{e_1,e_2=\varphi e_1,\xi\}$ such that
$$
  Qe_1=\lambda e_1,\qquad
  Qe_2=0,\qquad
  Q\xi=0.
$$

\item[\rm(ii)]  If $V\in\mathbb R\xi$, then $
  QX=\lambda\eta(X)\xi.
$
Thus the metric is $\eta$-Einstein.

\item[\rm(iii)]  If $V\notin D$ and $V\notin\mathbb R\xi$, then with respect to a
local $\varphi$-orthonormal frame $\{e_1,e_2=\varphi e_1,\xi\}$, one can
write
$$
  V=\sin\theta\,e_1+\cos\theta\,\xi,\qquad
  0<|\sin\theta|<1.
$$
Moreover, 
$$
  Q=
  \begin{pmatrix}
  \lambda\sin^2\theta & 0 & \lambda\sin\theta\cos\theta \\
  0 & 0 & 0 \\
  \lambda\sin\theta\cos\theta & 0 & \lambda\cos^2\theta
  \end{pmatrix},
$$
with respect to this frame.
\end{itemize}
\end{theorem}

\begin{proof}
Since the metric is weakly Einstein and not Einstein, by
Proposition~\ref{prop:3d-we-spectrum} we see that $Q$ has rank one on
the open subset $\mathcal{U}$.  Since $Q$ is symmetric and has rank one, it determines a unique
nonzero Ricci eigenspace and a two-dimensional kernel.  Hence locally there
exist a nonzero smooth function $\lambda$ and a local unit eigenvector
field $V$ such that
$$
  QV=\lambda V.
$$
The kernel of $Q$ is $V^\perp$, and therefore
$$
  QX=\lambda g(X,V)V
$$
for every vector field $X$ on $\mathcal{U}$.

We compare $V$ with $D=\ker\eta$ and $\mathbb R\xi$.  If
$V\in D$, choose $e_1=V$ and $e_2=\varphi e_1$.  Then
$\{e_1,e_2,\xi\}$ is a local $\varphi$-orthonormal frame.  Since
$e_2\perp V$ and $\xi\perp V$, we get
$$
  Qe_1=\lambda e_1,\qquad
  Qe_2=0,\qquad
  Q\xi=0.
$$

If $V\in\mathbb R\xi$ is a unit vector, then $V=\pm\xi$. Hence, $$ QX=\lambda g(X,\xi)\xi =\lambda\eta(X)\xi. $$  Therefore the metric is $\eta$-Einstein.

Finally, suppose that $V\notin D$ and $V\notin\mathbb R\xi$.  Then the
projection of $V$ onto $D$ is nonzero.  Choose a local unit vector field
$e_1\in D$ in the direction of this projection, and put
$e_2=\varphi e_1$.  Then $\{e_1,e_2,\xi\}$ is a local
$\varphi$-orthonormal frame, and
$$
  V=\sin\theta\,e_1+\cos\theta\,\xi
$$
for some smooth function $\theta$ with $0<|\sin\theta|<1$.  Substitution
into $QX=\lambda g(X,V)V$ gives
\begin{equation*}
    \begin{aligned}
        &Qe_1=\lambda\sin^2\theta\,e_1+
       \lambda\sin\theta\cos\theta\,\xi,\\
        &Qe_2=0,\\
        &Q\xi=\lambda\sin\theta\cos\theta\,e_1+
       \lambda\cos^2\theta\,\xi. 
    \end{aligned}
\end{equation*}
The above cases cover all possible positions of $V$ relative to
$D$ and $\mathbb R\xi$.
\end{proof}
\begin{remark}
In Theorem~\ref{thm:rank-one-we}, the third case is the only case in
which the weakly Einstein condition is compatible with $Q\xi\neq0$
while $\xi$ is not an eigenvector of $Q$. Thus, even though the
Ricci operator has rank one, the decomposition
$$
  TM=D\oplus\mathbb R\xi
$$
need not agree with the Ricci eigenspace decomposition. In particular, in the third case the metric is neither $\eta$-Einstein nor satisfies $Q\xi=0$.
\end{remark}

\section{Contact metric three-manifolds with $Q\xi=0$}
\label{sec:contact}
In this section, we prove that the condition $Q\xi=0$ has a much stronger consequence for contact metric three-manifolds than for almost contact metric three-manifolds.  On an almost contact metric three-manifold, the equation $Q\xi=0$ only says that the characteristic
vector field $\xi$ lies in the kernel of $Q$; it imposes no direct
restriction on the Ricci curvature along $D=\ker\eta$.  On a contact
metric manifold, however, the structure equations relate the horizontal
distribution $D$ to the Reeb direction $\mathbb R\xi$.  These
relations force the horizontal Ricci components to degenerate under the
assumption $Q\xi=0$.  Consequently, the Ricci operator has rank at most
one, and the metric is weakly Einstein.

\begin{theorem}
\label{thm:contact-Qxi-rigidity}
Let $M=(M^3,\varphi,\xi,\eta,g)$ be a contact metric three-manifold.  If {$M$ satisfies}
$Q\xi=0,$
then 
$$\operatorname{rank}Q\le 1.$$
Consequently, either the metric is flat, or it is weakly Einstein but not Einstein. 
\end{theorem}

\begin{proof}
Let $\mathcal{U}$ be a local open subset of $M$.  Choose a local $\varphi$-orthonormal frame
$
\{e_1,e_2=\varphi e_1,\xi\}
$
on $\mathcal{U}$ which is adapted to $h=\frac12\mathcal L_\xi\varphi$. 
Using \eqref{eq:contact} and \eqref{eq:structure_h},
the Lie brackets with respect to this frame have the form 
$$
[e_1,e_2]=2\xi+c e_1+d e_2,\qquad
[e_1,\xi]=a e_2,\qquad
[e_2,\xi]=b e_1,
$$
where $a,b,c,d$ are local smooth functions on $\mathcal{U}$.
Using the Koszul formula,
we obtain
\begin{equation*}
\begin{alignedat}{3}
\nabla_{e_1}e_1&=-c e_2,&\qquad
\nabla_{e_1}e_2&=c e_1+\frac{2-a-b}{2}\xi,&\qquad
\nabla_{e_1}\xi&=\frac{a+b-2}{2}e_2,\\[4pt]
\nabla_{e_2}e_1&=-d e_2-\frac{2+a+b}{2}\xi,&\qquad
\nabla_{e_2}e_2&=d e_1,&\qquad
\nabla_{e_2}\xi&=\frac{2+a+b}{2}e_1,\\[4pt]
\nabla_{\xi}e_1&=\frac{b-a-2}{2}e_2,&\qquad
\nabla_{\xi}e_2&=\frac{2+a-b}{2}e_1,&\qquad
\nabla_{\xi}\xi&=0.
 \end{alignedat}
\end{equation*}
A direct computation gives
\begin{equation*}
    \begin{aligned}
    &\rho(e_1,e_1)=\,e_1(d)-e_2(c)-c^2-d^2-2-2a+\frac{b^2-a^2}{2},\\
    &\rho(e_2,e_2)=\,
e_1(d)-e_2(c)-c^2-d^2-2+2b+\frac{a^2-b^2}{2},\\
&\rho(e_1,e_2)=\,-\frac12\xi(a+b),\\
&\rho(e_1,\xi)=(a+b)c+\frac12 e_2(a+b),\\
&\rho(e_2,\xi)=-(a+b)d+\frac12 e_1(a+b),\\
&\rho(\xi,\xi)=\,2-\frac12(a+b)^2.
    \end{aligned}
\end{equation*}
Now assume that $Q\xi=0$. Then
we have
$
\rho(X,\xi)=0
$
for every vector field $X$.  In particular,
$$
0=\rho(\xi,\xi)=2-\frac12(a+b)^2.
$$
Hence
$
(a+b)^2=4.
$
After restricting $\mathcal{U}$ further, if necessary, we may assume that either
$$
a+b=2\quad \text{or}\quad
a+b=-2.
$$
In both cases $a+b$ is locally constant.  Therefore
we obtain
$$(a+b)c=0,\qquad (a+b)d=0.
$$
Since $a+b=\pm2\neq0$, it follows that
$
c=d=0.
$

Suppose first that $a+b=2$.  Then we obtain
$$
\rho(e_1,e_1)=-4a,\qquad
\rho(e_2,e_2)=0,\qquad
\rho(e_1,e_2)=0.
$$
Thus, with respect to the frame $\{e_1,e_2,\xi\}$, the Ricci operator has the form
$$
Q=
\begin{pmatrix}
-4a&0&0\\
0&0&0\\
0&0&0
\end{pmatrix}.
$$

Similarly, suppose that $a+b=-2$.  Then we have
$$
\rho(e_1,e_1)=0,\qquad
\rho(e_2,e_2)=
-8-4a,\qquad
\rho(e_1,e_2)=0.
$$
Therefore, with respect to $\{e_1,e_2,\xi\}$,
$$
Q=
\begin{pmatrix}
0&0&0\\
0&-8-4a&0\\
0&0&0
\end{pmatrix}.
$$

From the above Ricci operators, we see that $\operatorname{rank}Q\le1$.  If $Q=0$, then the metric is flat.  If $\operatorname{rank}Q=1$,  by Proposition \ref{prop:3d-we-spectrum}, the metric is weakly Einstein but not Einstein.  This completes the proof.
\end{proof}
By Theorem~\ref{thm:contact-Qxi-rigidity}, we immediately  have the following corollary:
\begin{corollary}
\label{cor:no-contact-rank-two}
There is no contact metric manifold of dimension three satisfying
$$Q\xi=0,   \qquad   \operatorname{rank}Q=2. $$
\end{corollary}
For comparison, we recall the classification of weakly $\eta$-Einstein contact metric three-manifolds obtained in \cite[Theorem~3.4]{CE25}. This classification provides an independent consistency check for Theorem~\ref{thm:contact-Qxi-rigidity} and clarifies the possible forms of the Ricci operator under the condition $Q\xi=0$.
\begin{proposition}
[Theorem~3.4 in \cite{CE25}] \label{prop:contact-weakly-eta-einstein}
Let $M=(M,\varphi,\xi,\eta,g)$ be a contact metric three-manifold. If $M$ is weakly $\eta$-Einstein, then $M$ is one of the
following:
\begin{enumerate}
\item[\rm(i)]  it is Sasakian;
\item[\rm(ii)]  it is flat;
\item[\rm(iii)] it is a space of constant $\varphi$-sectional curvature
$\mu^2-1$ and $\xi$-sectional curvature $1-\mu^2$, where 
$\mu$ is the eigenvalue function of 
$h=\frac12\mathcal L_\xi\varphi
\text{ on } D=\ker\eta$;

\item [\rm(iv)] it has constant $\varphi$-sectional curvature $-2\nu$, and its
$\xi$-sectional curvature depends on the horizontal direction. More
precisely, with respect to a local $\varphi$-orthonormal frame
$\{e_1,e_2=\varphi e_1,\xi\}$,
$$
K(e_1,\xi)=2\nu,\qquad K(e_2,\xi)=-2\nu,
$$
and, for any unit vector field
$
X=\cos\theta\,e_1+\sin\theta\,e_2\in D,
$
one has
$
K(X,\xi)=2\nu\cos 2\theta.
$
Hence $K(X,\xi)$ ranges between $-2\nu$ and $2\nu$. Here
$
\nu=g(\nabla_\xi e_2,e_1)
$
is the connection coefficient of the chosen local $\varphi$-orthonormal
frame.
\end{enumerate}
\end{proposition}

\begin{remark}\label{rem:comparison-prop55}
Proposition~\ref{prop:contact-weakly-eta-einstein} is compatible with 
Theorem~\ref{thm:contact-Qxi-rigidity}.  Indeed, in the Sasakian case one has
$Q\xi=2\xi$, and hence the condition $Q\xi=0$ cannot occur.  In the
flat case, $Q=0$.  In the third case,
the Ricci operator satisfies
$$
Qe_1=0,\qquad Qe_2=0,\qquad Q\xi=2(1-\mu^2)\xi .
$$
Thus $Q\xi=0$ implies $\mu^2=1$, and consequently $Q=0$.  In the
fourth case, one obtains
$$
Qe_1=0,\qquad Qe_2=-4\nu e_2,\qquad Q\xi=0,
$$
so that $\operatorname{rank}Q\le 1$.

Therefore, among the weakly $\eta$-Einstein contact metric
three-manifolds listed in Proposition~\ref{prop:contact-weakly-eta-einstein},
there is no case satisfying $Q\xi=0$ and $\operatorname{rank}Q=2$.
This is consistent with Corollary~\ref{cor:no-contact-rank-two}.
\end{remark}
\section{Examples}\label{sec:examples}
In this section, we give examples of almost contact metric three-manifolds satisfying  $Q\xi=0$ (with $\operatorname{rank}Q=1$ or $\operatorname{rank}Q=2$).

\begin{example}[The case $Q\xi=0$ and $\operatorname{rank}Q=1$]
\label{ex:rank-one-qxi-zero}
Let
$$
  M=(0,\infty)\times\mathbb R^2
$$
with coordinates $(t,x,y)$. We regard $M$ as a warped product with
one-dimensional base $(0,\infty)$ and two-dimensional flat fiber
$\mathbb R^2$. Let
$$
  g=dt^2+f(t)^2(dx^2+dy^2),\qquad f(t)=\sqrt t.
$$
With respect to this metric, we take
$$
  e_1=\frac{\partial}{\partial t},\qquad
  e_2=\frac1{\sqrt t}\frac{\partial}{\partial x},\qquad
  \xi=\frac1{\sqrt t}\frac{\partial}{\partial y}.
$$
Then $\{e_1,e_2,\xi\}$ is a local orthonormal frame. Put
$
  \eta=g(\xi,\cdot)=\sqrt t\,dy
$ and define a $(1,1)$-tensor field
$\varphi$ by
$$
  \varphi e_1=e_2, \qquad
  \varphi e_2=-e_1,  \qquad
  \varphi\xi=0.
$$
Then $(\varphi,\xi,\eta,g)$ is an almost contact metric structure. But  it is not a contact metric structure. Indeed,  
$$ d\eta=\frac{1}{4\sqrt{t}}\,dt\wedge dy, $$ 
and hence $ d\eta(e_1,e_2)=0. $ On the other hand, $ g(e_1,\varphi e_2)= -1. $ Therefore, $$ d\eta(e_1,e_2)\neq g(e_1,\varphi e_2). $$ 
For this warped product, the standard Ricci curvature formulas give
$$
  \rho(e_1,e_1)
  =
  -2\frac{f''}{f}
  =
  \frac1{2t^2},\qquad  \rho(e_2,e_2)=\rho(\xi,\xi)
  =
  -\frac{ff''+(f')^2}{f^2}
  =
  0,
$$
all others being zero with respect to
$\{e_1,e_2,\xi\}$. Hence the Ricci operator is diagonal with
respect to this frame, that is,
$$
  Qe_1=\frac1{2t^2}e_1,\qquad
  Qe_2=0,\qquad
  Q\xi=0.
$$
Consequently, $Q\xi=0$ and $\operatorname{rank}Q=1.$ From Proposition~\ref{prop:3d-we-spectrum}, the metric is weakly Einstein.
\end{example}

\begin{example}[The case $Q\xi=0$ and $\operatorname{rank}Q=2$]
\label{ex:rank-two-qxi-zero}
Let $M=\mathbb R^3$ with coordinates $(t,x,y)$, and let
$$
  g=dt^2+e^{2t}dx^2+(1+e^{-t})^2dy^2.
$$
With respect to this metric, we take
$$
  e_1=\frac{\partial}{\partial t},\qquad
  e_2=e^{-t}\frac{\partial}{\partial x},\qquad
  \xi=\frac1{1+e^{-t}}\frac{\partial}{\partial y}.
$$
Then $\{e_1,e_2,\xi\}$ is a local orthonormal frame. We put
$$
  \eta=g(\xi,\cdot)=(1+e^{-t})dy,
$$
and define a $(1,1)$-tensor field
$\varphi$ by
$$
  \varphi e_1=e_2,\qquad
  \varphi e_2=-e_1,\qquad
  \varphi\xi=0.
$$
Then $(\varphi,\xi,\eta,g)$ is an almost contact metric structure. But it is not a contact metric structure. Indeed, 
$$d\eta=-\frac12 e^{-t}\,dt\wedge dy.$$
and hence $ d\eta(e_1,e_2)=0. $ On the other hand, $ g(e_1,\varphi e_2)= -1. $ Therefore, $$ d\eta(e_1,e_2)\neq g(e_1,\varphi e_2). $$
By a direct computation, we obtain the Ricci  operator with respect to $\{e_1,e_2,\xi\}$:
$$
  Qe_1=-\frac{1+2e^{-t}}{1+e^{-t}}e_1,\qquad
  Qe_2=-\frac1{1+e^{-t}}e_2,\qquad
  Q\xi=0.
$$
This yields $Q\xi=0$ and $\operatorname{rank}Q=2$. Moreover, the eigenvalues of $Q$ with respect to
$\{e_1,e_2,\xi\}$ are
$$
  \lambda_1=-\frac{1+2e^{-t}}{1+e^{-t}},\qquad
  \lambda_2=-\frac1{1+e^{-t}},\qquad
  \lambda_3=0.
$$
By \eqref{eq:breve-difference}, taking $i=1$, $j=2$, and $k=3$,
we obtain
$
  \breve R(e_1,e_1)=\breve R(e_2,e_2).
$
Since $Q$ is diagonal with respect to
$\{e_1,e_2,\xi\}$, formula  \eqref{eq:3-d_breve}  shows that all off-diagonal
components of $\breve R$ vanish.
Hence $\breve R$ has the form
$
  \breve R=\alpha g+\beta\,\eta\otimes\eta
$
for suitable smooth functions $\alpha$ and $\beta$.  Thus the almost
contact metric manifold is weakly $\eta$-Einstein. Since $\lambda_1\ne\lambda_2$, it is neither $\eta$-Einstein nor  weakly Einstein. 

\end{example}

These examples show that, for almost contact metric three-manifolds,
the condition $Q\xi=0$ does not imply $\operatorname{rank}Q\leq1$. This means  that the contact metric assumption in Theorem~\ref{thm:contact-Qxi-rigidity} is essential. Without this assumption, the conclusion  no longer holds. 

\section{Weakly Einstein homogeneous contact metric three-manifolds}
\label{sec:homogeneous}
In this section, we determine the simply connected homogeneous contact
metric three-manifolds whose metrics are weakly Einstein. Perrone's
description of homogeneous contact metric three-manifolds shows that
each such manifold is a Lie group endowed with a left-invariant contact
metric structure ~\cite{Perrone1998}. The corresponding Lie algebras are divided into the unimodular and nonunimodular cases.

\begin{theorem}\label{thm:homogeneous-classification}
Let $(M^3,\varphi,\xi,\eta,g)$ be a simply connected homogeneous contact
metric  three-manifold.  Then $M$ is weakly Einstein if and
only if $M$ is unimodular and, with respect to a suitable orthonormal
$\varphi$-basis
$$
  \{e_1,e_2=\varphi e_1,e_3=\xi\},
$$
the Lie algebra is given by
$$
  [e_1,e_2]=2e_3,\qquad
  [e_2,e_3]=ae_1,\qquad
  [e_3,e_1]=be_2,
$$
where the structure constants $a,b$ satisfy
$$
  a+b=2,\qquad
  a-b=\pm2,\qquad
  \text{or}\qquad
  (a,b)=(2,2).
$$
Consequently, every weakly Einstein homogeneous contact metric three-manifold is locally isometric to one of the unimodular Lie groups 
$SU(2)$, $\widetilde{SL}(2,\mathbb R)$, $\widetilde{E}(2)$, $E(1,1)$  with a left-invariant contact metric structure.

\end{theorem}

\begin{proof}
By Perrone's theorem~\cite{Perrone1998}, every simply connected
homogeneous contact metric three-manifold is a Lie group endowed with a
left-invariant contact metric structure.  The corresponding Lie groups
are either unimodular or nonunimodular.

 We first consider the unimodular case. Let  $G$  be a simply connected
unimodular Lie group endowed with a left-invariant contact metric
structure. Then, with respect to a suitable orthonormal
$\varphi$-basis $\{e_1,e_2=\varphi e_1,e_3=\xi\}$, the Lie algebra is given by
$$ [e_1,e_2]=2e_3,\qquad  [e_2,e_3]=ae_1,\qquad  [e_3,e_1]=be_2.$$
For this Lie algebra, the Ricci operator is diagonal with respect to the
same basis \cite{CE23}:
$$  Qe_1=\lambda_1e_1,\qquad  Qe_2=\lambda_2e_2,\qquad  Qe_3=\lambda_3e_3, $$
where
\begin{equation}\label{eq:uni-eignevalues}
    \begin{aligned}
        \lambda_1 &  =  \frac{(a-b+2)(a+b-2)}2,\\
         \lambda_2 & =  -\frac{(a-b-2)(a+b-2)}2,\\
           \lambda_3 &  =  -\frac{(a-b-2)(a-b+2)}2.
    \end{aligned}
\end{equation} 
By Proposition~\ref{prop:3d-we-spectrum}, a three-dimensional
Riemannian metric is weakly Einstein if and only if either it is
Einstein or its Ricci operator $Q$ has rank one. We therefore divide the
analysis into two parts: the Einstein case and the case in which $Q$ has
rank one.

We first determine when $Q$ has rank one.  Since $Q$ is diagonal, this
means that exactly two of $\lambda_1,\lambda_2,\lambda_3$ vanish.  The
zero sets are
$$
\begin{aligned}
  \lambda_1=0
  &\quad\Longleftrightarrow\quad
  a-b=-2\quad\text{or}\quad a+b=2,\\
  \lambda_2=0
  &\quad\Longleftrightarrow\quad
  a-b=2\quad\text{or}\quad a+b=2,\\
  \lambda_3=0
  &\quad\Longleftrightarrow\quad
  a-b=\pm2.
\end{aligned}
$$
Thus $\lambda_1=\lambda_2=0$ is equivalent to
$a+b=2.$
Similarly,
\begin{equation*}
    \begin{aligned}
    \lambda_1=\lambda_3=0&
  \quad\Longleftrightarrow\quad
  a-b=-2
  \quad\text{or}\quad
  (a,b)=(2,0), \\
  \lambda_2=\lambda_3=0&
  \quad\Longleftrightarrow\quad
  a-b=2
  \quad\text{or}\quad
  (a,b)=(0,2).
\end{aligned}
\end{equation*}
At the two exceptional points $(a,b)=(2,0)$ and $(a,b)=(0,2)$,
all three eigenvalues vanish.  Thus $Q=0$ at these points, and the
corresponding metrics are flat.  Therefore the Ricci
operator has rank one precisely when
$$  a+b=2  \quad\text{or}\quad  a-b=\pm2,$$
excluding  $(a,b)=(2,0)$ and
$(a,b)=(0,2)$.

It remains to determine the Einstein cases. Solving
$\lambda_1=\lambda_2=\lambda_3$, we first obtain from
$\lambda_1=\lambda_2$ that
$$ (a+b-2)(a-b)=0.$$
Hence either $a+b=2$ or $a=b$.  

If $a+b=2$, then
$\lambda_1=\lambda_2=0$. For the metric to be
Einstein, one must also have
$\lambda_3=0$.  This gives $a-b=\pm2$, and hence
$$  (a,b)=(2,0)\quad\text{or}\quad (a,b)=(0,2).$$
At both points, $Q=0$, so the corresponding metrics are  flat.

If $a=b$, then
$$
  \lambda_1=\lambda_2=2a-2,\qquad
  \lambda_3=2.
$$
Thus the metric is Einstein if and only if $2a-2=2$, namely $a=2$.  Since
$a=b$, we obtain
$$
  (a,b)=(2,2).
$$
This gives the non-flat Einstein case.

Consequently, in the unimodular homogeneous contact metric case, the
weakly Einstein condition is equivalent to
$$
  a+b=2,\qquad
  a-b=\pm2,\qquad
  \text{or}\qquad
  (a,b)=(2,2).
$$

We next consider the nonunimodular case. Then, with respect to a suitable $\varphi$-basis $\{e_1,e_2=\varphi e_1,e_3=\xi\}$,  the Lie algebra is given by
$$  [e_1,e_2]=ce_2+2e_3,\qquad  [e_2,e_3]=0,\qquad  [e_3,e_1]=de_2,\qquad  c\neq0.$$
The Ricci operator takes the form
$$
  Qe_1=\mu_1e_1,\qquad
  Qe_2=\mu_2e_2+cde_3,\qquad
  Qe_3=cde_2+\mu_3e_3,
$$
where
$$
  \mu_1=-c^2-2+2d-\frac{d^2}{2},\qquad
  \mu_2=-c^2-2+\frac{d^2}{2},\qquad
  \mu_3=2-\frac{d^2}{2}.
$$
Since $Qe_2,Qe_3\in \operatorname{span}\{e_2,e_3\}$, the subspace
$\operatorname{span}\{e_2,e_3\}$ is $Q$-invariant.  Hence the restriction
of $Q$ to this subspace is represented by
$$
  A=
  \begin{pmatrix}
  \mu_2 & cd \\
  cd & \mu_3
  \end{pmatrix}.
$$
Its determinant is
$$
\begin{aligned}
  \det A
  &=
  -\frac{(d^2-4)^2+2c^2(d^2+4)}{4}.
\end{aligned}
$$
Since $c\neq0$, we have
$
  (d^2-4)^2+2c^2(d^2+4)>0.
$
Thus $\det A<0$, and in particular $A$ is nonsingular.  Hence
$
  \operatorname{rank}A=2,
$
and therefore
$$
  \operatorname{rank}Q\ge 2.
$$
By Proposition~\ref{prop:3d-we-spectrum}, the non-Einstein weakly Einstein case cannot occur in the
nonunimodular case.

Now suppose that the metric is
Einstein. Then the
off-diagonal entry of the matrix $A$ must vanish, which gives $cd=0$.
Since $c\neq0$, this implies $d=0$.  Substituting $d=0$ into the Ricci
operator gives
$$   Qe_1=(-c^2-2)e_1,\qquad Qe_2=(-c^2-2)e_2,\qquad  Qe_3=2e_3.$$
Since the metric is Einstein, $c^2=-4$. This is a contradiction. Therefore no simply connected nonunimodular homogeneous contact metric
three-manifold has a weakly Einstein metric. 
\end{proof}

\begin{remark}\label{rem:homogeneous-table}
The weakly Einstein cases in Theorem~\ref{thm:homogeneous-classification}
can be described in terms of the Ricci operator $Q$. In the unimodular
case, the Ricci eigenvalues are given by \eqref{eq:uni-eignevalues}.
Thus the result can be summarized as follows:
\renewcommand{\arraystretch}{1.15}
\begin{center}
\begin{tabular}{
  @{}
  L{0.30\textwidth}
  C{0.18\textwidth}
  C{0.27\textwidth}
  C{0.18\textwidth}
  @{}
}
\toprule
Condition on $(a,b)$
&
\makecell[c]{Ricci\\ operator}
&
\makecell[c]{Geometric\\ description}
&
Lie groups
\\
\midrule
\midrule

$(a,b)=(2,2)$
&
$Q=2I$
&
non-flat Einstein
&
$SU(2)$
\\
\midrule

$(a,b)=(2,0)$ or $(0,2)$
&
$Q=0$
&
flat
&
$\widetilde{E}(2)$
\\
\midrule

\makecell[l]{
  $a+b=2,$\\[0.8ex]
  $(a,b)\notin\{(2,0),(0,2)\}$
}
&
$\operatorname{rank}Q=1$
&
\makecell[c]{
  weakly Einstein,\\[0.8ex]
  not Einstein
}
&
\makecell[c]{
  $SU(2),$\\[0.8ex]
  $\widetilde{\operatorname{SL}}(2,\mathbb{R})$
}
\\
\midrule

\makecell[l]{
  $a-b=\pm2,$\\[0.8ex]
  $(a,b)\notin\{(2,0),(0,2)\}$
}
&
$\operatorname{rank}Q=1$
&
\makecell[c]{
  weakly Einstein,\\[0.8ex]
  not Einstein
}
&
\makecell[c]{
  $SU(2),$\\[0.8ex]
  $\widetilde{\operatorname{SL}}(2,\mathbb{R}),$\\[0.8ex]
  $E(1,1)$
}
\\
\bottomrule
\end{tabular}
\end{center}
\end{remark}

\begin{remark}
The Heisenberg group $H_3$ with its standard Sasakian
structure is not weakly Einstein. Indeed, in the notation of
Theorem~\ref{thm:homogeneous-classification}, it corresponds to
$(a,b)=(0,0)$. Hence
$$
(\lambda_1,\lambda_2,\lambda_3)=(-2,-2,2).
$$
Thus $Q$ is neither a scalar multiple of the identity nor of rank one.
By Proposition~\ref{prop:3d-we-spectrum}, the metric is not weakly
Einstein.
\end{remark}

\section*{Acknowledgements}

The authors would like to express their sincere gratitude to Professor Jong Taek Cho for his insightful comments and valuable suggestions, which helped to improve this work.

\end{document}